\documentclass[secnum,leqno]{rims-bessatsu}%%% eqno resets every section 
\usepackage{amssymb, amsmath}%%% option packages
\usepackage{graphics}%%% option packages
\usepackage[dvips]{graphicx}%%% option packages
\usepackage{eepic}%%% option packages
\usepackage{epic}%%% option packages
\newtheorem{thm}{Theorem}[section]

\newtheorem{lmm}[thm]{Lemma}

\newtheorem{prp}[thm]{Proposition}

\theoremstyle{definition}

\theoremstyle{remark}
\newtheorem*{rem}{Remark}%%% * means no numbering
\title{An elementary proof of the Voros connection formula for WKB solutions to the Airy
equation with a large parameter}
\TitleHead{An elementary proof of the Voros connection formula for the Airy
equation}          %optional
\dedicatory{Dedicated to Professor Yoshitsugu Takei on his 60th birthday}           %optional
\author{Takashi \textsc{Aoki}\footnote{Kindai University, Kowakae 3-4-1, Higashi-Osaka, 
577-8502, Japan.\newline e-mail: \texttt{aoki@math.kindai.ac.jp}},
          Takao \textsc{Suzuki}\footnote{Kindai University, Kowakae 3-4-1, Higashi-Osaka, 
577-8502, Japan.\newline e-mail: \texttt{suzuki@math.kindai.ac.jp}}
~and Shofu \textsc{Uchida}\footnote{Kindai University, Kowakae 3-4-1, Higashi-Osaka, 
577-8502, Japan.\newline e-mail: \texttt{uchida@math.kindai.ac.jp }}}
\AuthorHead{Aoki, T., Suzuki, T. and Uchida, S.}         %optional but recommended
\support{The first and the second author are supported by JSPS KAKENHI Grant No. 18K03385.}  %optional
         %optional
\VolumeNo{x}           %editors must define this.
\YearNo{202x}           %editors must define this.
\PagesNo{000--000}      %editors must define this.
\communication{Received April 20, 202x.
Revised September 11, 202x.}      %editors must define this.
\begin{document}
%
% The text goes here.  
% Be sure to use the appropriate "theorem-like" environment as 
% is the following examples.  Never use plain TeX commands for these, as
% they will cause interference with the styles of other papers. 

\maketitle

%\tableofcontents      %optional
\begin{abstract}      %optional
The Voros connection formula for WKB solutions to the Airy equation with a large 
parameter is proved by using cubic equations. 
Some parts of the results are generalized to the Pearcey system, which is a two-variable
version of the Airy equation, is given.
\end{abstract}

\section{Introduction}
In \cite{akt0}, the Voros connection formula (cf. \cite{V}) for the WKB solutions to Schr\"odinger-type ordinary
differential equations is proved from the viewpoint of microlocal analysis \cite{K3}. The proof consists of
two parts. In the first part, the Schr\"odinger-type ordinary differential equation with an analytic potential 
is formally transformed to
the Airy equation with a large parameter near a simple turning point. The formal series appearing
in the transformation can be justified by using microdifferential operators. In the second part, 
the Voros connection formula for the Airy equation is proved by computing the Borel transform
of the WKB solutions directly in terms of the Gauss hypergeometric functions.
The classical connection formulas for the hypergeometric functions yield the Voros formula
for the Airy equation. Combining these two parts, we have the Voros formula for general
equations. (See \cite{KT} also.)

In this article, we focus on the second part. We observe that the parameters in the hypergeometric functions
used in the proof in \cite{akt0, KT} are contained in the Schwarz list \cite{Sch}. In other words, they are algebraic functions.
Our aim is to prove the connection formula for the WKB solutions to the Airy equation by using analytic continuation of
algebraic functions, not of the hypergeometric functions. This point of view gives some generalization. 
We find a similar structure in the Pearcey system (cf. \cite{OK}) with a large parameter. This system is a natural
generalization of the Airy equation to the two-variable case. 
We see that the Borel transforms of the suitably normalized WKB solutions to the
Pearcey system are also algebraic functions. Hence such WKB solutions
are resurgent. We hope that this example provides a part of basics of the prospective theory for the 
exact WKB analysis of holonomic systems.

\section{The Airy equation with a large parameter and its WKB solutions}
The differential equation
\begin{equation}\label{2:Airy-0}
\left(-\frac{d^{2}}{dz^{2}}+z\right)w=0
\end{equation}
is called the Airy equation (cf. \cite{O}). 
Airy used a solution to this equation expressed by an integral, known as the ``Airy function'' nowadays (see \S~3), 
effectively in his theory of  the rainbow (caustics)
\cite{Ai}.
We introduce a positive large parameter $\eta$ by setting
$z=\eta^{2/3}x, \, \psi(x,\eta)=w(\eta^{2/3}x)$. Then we have 
a differential equation of the form
\begin{equation}\label{2:Airy-1}
\left(-\frac{d^{2}}{dx^{2}}+\eta^{2}x \right)\psi=0.
\end{equation}
We call this the Airy equation with the large parameter, or simply, the Airy equation. 
This equation plays a fundamental role in the exact WKB analysis.
Let $S$ denote the logarithmic derivative
of the unknown function $\psi$, namely
$\displaystyle
S=\frac{d}{dx}\log\psi.
$
Then $S$ should satisfy the following Riccati-type equation:
\begin{equation}\label{2:Riccati}
\frac{dS}{dx}+S^{2}=\eta^{2}x.
\end{equation}
This equation has a formal solution $\displaystyle S=\sum_{j=-1}^{\infty}\eta^{-j}S_{j}$
defined by the recurrence relation
\begin{equation}\label{2:recur}
\left\{
\begin{split}
\ \ S_{-1}^{2}&=x,\\
S_{j+1}&=-\frac{1}{2S_{-1}}\left(\frac{dS_{j}}{dx}+\sum_{k=0}^{j}S_{k}S_{j-k}\right)
\quad (j=-1,0,1,2,3,\dots).
\end{split}
\right.
\end{equation}
We consider the exponential of the integral of $S$:
\[
\psi=\exp\left(\int Sdx\right),
\]
which formally satisfies \eqref{2:Airy-1}. We call this a WKB solution to \eqref{2:Airy-1}. 
We easily see that
\[
S_{-1}=x^{1/2},\ S_{0}=-\frac{1}{4x}, \,\dots\, 
\]
and if we fix the branch of the square root, say, as $x^{1/2}>0$ for $x>0$, we have a formal solution
$S^{(+)}$. Another choice of the branch also gives 
a formal solution $S^{(-)}$.
If we set 
\[
S_{\rm odd}=\sum_{j=-1}^{\infty}\eta^{-2j-1}S_{2j+1}=\eta x^{1/2}-\eta^{-1}\frac{5}{32}x^{-5/2}-\eta^{-3}\frac{1105}{2048}x^{-11/2}-
\cdots,
\]
\[
S_{\rm even}=\sum_{j=0}^{\infty}\eta^{-2j}S_{2j}=-\frac{1}{4x}-\eta^{-2}\frac{15}{64}x^{-4}-\frac{1695}{1024}x^{-7}+\cdots,
\]
then we have
\[
S^{(\pm)}=\pm S_{\rm odd}+S_{\rm even}
\]
satisfies \eqref{2:Riccati} and
\[
S_{\rm even}=-\frac{1}{2}\frac{d}{dx}\log S_{\rm odd}
\]
holds. Hence we may take $-1/2\log S_{\rm odd}$ as a primitive of $S_{\rm even}$ and have 
the special WKB solutions of the form
\begin{equation}\label{2:wkbsol-1}
\psi_{\pm}=\frac{1}{\sqrt{S_{\rm odd}}}\exp\left(\pm\int_{0}^{x}S_{\rm odd}dx\right).
\end{equation}
Here the integral is defined by one half of the term-by-term contour integral
of $S_{\rm odd}$ starting from $x$ on the second sheet of the Riemann surface of $\sqrt{x}$, 
going around the origin counterclockwise and back to the $x$ on the first sheet:
\[
\int_{0}^{x}S_{\rm odd}dx=\eta\frac{2}{3} x^{3/2}+\eta^{-1}\frac{5}{48}x^{-3/2}+
\eta^{-3}\frac{1105}{9216}x^{-9/2}+\cdots.
\]
We observe that $S_{\rm odd}$ has the weighted homogeneity
\[
S_{\rm odd}(\lambda^{2}x,\lambda^{-3}\eta)=\lambda^{-2}S_{\rm odd}(x,\eta).
\]
Hence $\psi_{\pm}$ satisfies
\[
\psi_{\pm}(\lambda^{2}x,\lambda^{-3}\eta)=\lambda\psi_{\pm}(x,\eta).
\]
This relation and \eqref{2:Airy-1} imply that $\psi_{\pm}$ are formal solutions to
the following system of partial differential equations in the variable $(x, \eta)$:
\begin{equation}\label{2:Airy-system}
\left\{
\begin{split}
&\ \ \left(-\frac{\partial^{2}}{\partial x^{2}}+\eta^{2}x\right)\psi=0,\\
&\ \ \left(2x\frac{\partial}{\partial x}-3\eta\frac{\partial}{\partial \eta}-1\right)\psi=0.
\end{split}\right.
\end{equation}
%%%%%%%%%%%%%%%%%%%%%%%%%%%%

\section{Integral representation}
There are two standard solutions to \eqref{2:Airy-0}, which are called the Airy functions \cite{O}:
\begin{align}\label{3:Ai-def}
{\rm Ai}(z)&=\frac{1}{2\pi i}\int_{\gamma_{1}}\exp\left(-z\xi+\frac{\xi^{3}}{3}\right)d\xi,\\ \label{3:Bi-def}
{\rm Bi}(z)&=\frac{1}{2\pi}\int_{-\gamma_{2}+\gamma_{3}}\exp\left(-z\xi+\frac{\xi^{3}}{3}\right)d\xi.
\end{align}
Here the paths $\gamma_{j}\ (j=1,2,3)$ are taken as shown in Fig. 3.1. 
More precisely, the arguments of three half-line asymptotes of the paths are $\pm\pi/3, \pi$.
\begin{center}
\includegraphics[width=50mm]{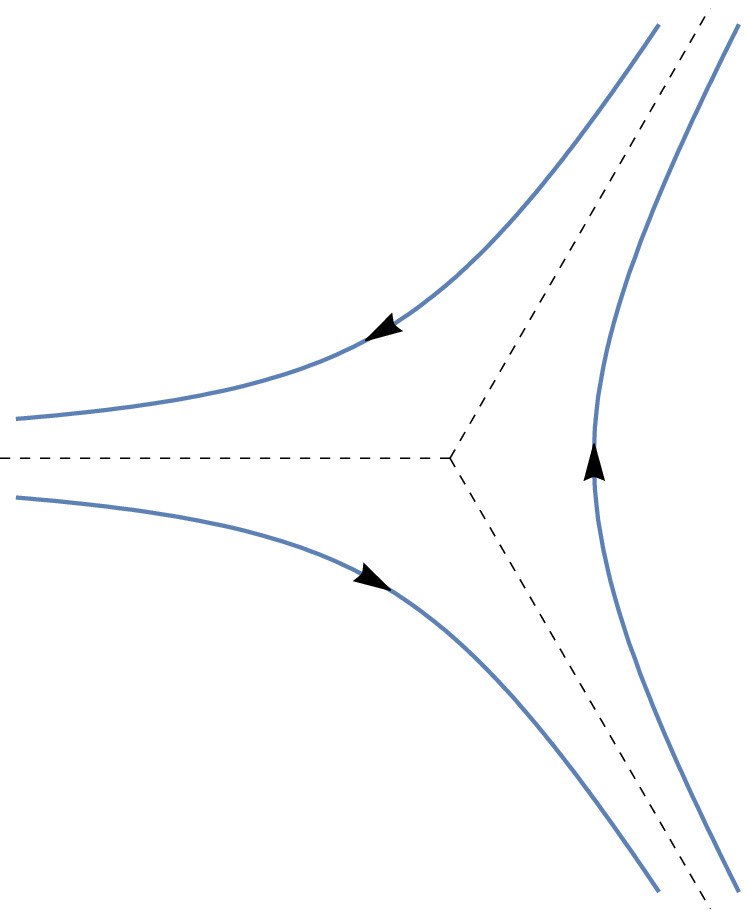}\\
Figure 3.1.
\end{center}

\vspace{-50mm}
\hspace*{53mm}$\gamma_{2}$

\vspace{5mm}
\hspace*{74mm}$\gamma_{1}$

\vspace{5mm}
\hspace*{52mm}$\gamma_{3}$

\vspace{25mm}
\noindent Hence the integrals
\begin{equation}\label{3:varphij}
\varphi_{j}=\int_{\gamma_{j}}\exp\left(\eta\left(\frac{t^{3}}{3}-xt\right)\right)dt\quad(j=1,2,3)
\end{equation}
are solutions to \eqref{2:Airy-1} satisfying $\varphi_{1}+\varphi_{2}+\varphi_{3}=0$.
Here $\gamma_{j}$'s are suitably deformed. 
We can see that $\varphi_{j}$ have the same homogeneity as that of $\psi_{\pm}$, namely,
\[
\varphi_{j}(\lambda^{2}x,\lambda^{-3}\eta)=\lambda\varphi_{j}(x,\eta).
\]
This implies $\psi=\varphi_{j}$ ($j=1,2,3$) satisfy \eqref{2:Airy-system}.
\begin{rem}
The Airy functions are expressed in terms of $\varphi_{j}$ as
\begin{equation}\label{3:AiBi-varphi}
{\rm Ai}(\eta^{2/3}x)=\frac{\eta^{1/3}}{2\pi i}\varphi_{1}(x,\eta),\quad
{\rm Bi}(\eta^{2/3}x)=\frac{\eta^{1/3}}{2\pi}(\varphi_{3}(x,\eta)-\varphi_{2}(x,\eta)).
\end{equation}
\end{rem}
We rewrite \eqref{3:varphij} by setting $t^{3}/3-xt=-y$:
\begin{equation}\label{3:int-1}
\varphi_{j}=\int_{c_{j}}g(x,y)\exp(-y\eta )dy.
\end{equation}
Here $c_{j}$ is the image of $\gamma_{j}$ by the mapping $t\mapsto y$ and
\[
g(x,y)=\left.\frac{1}{x-t^{2}}\right|_{t=t(x,y)}.
\]
The branch of the root $t=t(x,y)$ of the cubic equation 
\begin{equation}\label{3:eqfort}
t^{3}/3-xt=-y
\end{equation} 
is taken suitably (see \S \ref{6:section}). We note that \eqref{3:int-1} looks like the Laplace integral
defining the Borel sum of WKB solutions (see \eqref{5:PsiJ}) and hence $g$ is expected to have some
relation with the Borel transform of WKB solutions.
\begin{lmm}\label{3:Lemma1}
The function $g=g(x,y)$ defined as above satisfies the cubic equation
\begin{equation}\label{3:cubic1}
(9y^{2}-4 x^{3})g^{3}+3 x g+1=0.
\end{equation}
\end{lmm}
\begin{proof}
Eliminating $t$ from the relations
\[
(x-t^{2})g=1,\quad \frac{t^{3}}{3}-xt=-y,
\]
we have \eqref{3:cubic1}.
\end{proof}
\begin{prp}\label{3:sys-for-g}
The algebraic function $g$ defined by the cubic equation \eqref{3:cubic1}
satisfies the system of
partial differential equations
\begin{equation}\label{3:sysB2}
\left\{
\begin{split}
\ \ &\left(-\frac{\partial^{2}}{\partial x^{2}}+x\frac{\partial^{2}}{\partial y^{2}}\right)g=0,\\
\ \ &\left(2x\frac{\partial}{\partial x}+3y\frac{\partial}{\partial y}+2\right)g=0.
\end{split}\right.
\end{equation}
\end{prp}
\begin{proof}
We can express
$g_{x}$, $g_{xx}$, $g_{y}$ and $g_{yy}$ in terms of $g$ by differentiating \eqref{3:cubic1} twice in $x$ and in $y$.
Putting them into the left-hand sides of \eqref{3:sysB2} and using \eqref{3:cubic1},
we see they vanish.
\end{proof}
We observe that the characteristic variety of the system \eqref{3:sysB2} is the conormal
bundle of the curve $4x^{3}-9 y^{2}=0$ and hence the system is holonomic. 
The second equation of \eqref{3:sysB2} implies that the unknown function can be written in the form
\[
g(x,y)=\frac{1}{x}h\left(\frac{y}{x^{3/2}}\right)
\]
by using a function $h$ of one variable. It is easy to find a second-order ordinary differential equation for $h$
by using the first equation.
This implies the rank of the holonomic system equals 2. Hence we have
\begin{thm}\label{3:thm-for-g}
Let $g_{j}\ (j=1,2,3)$ be three branches of the algebraic function $g$ defined
by the cubic equation \eqref{3:cubic1}. 
Then any two of them form a basis of the analytic solution space of the holonomic system
\begin{equation}\label{3:sys-for-g}
\left\{\ 
\begin{split}
\ \ &\left(-\frac{\partial^{2}}{\partial x^{2}}+x\frac{\partial^{2}}{\partial y^{2}}\right)u=0,\\
\ \ &\left(2x\frac{\partial}{\partial x}+3y\frac{\partial}{\partial y}+2\right)u=0.
\end{split}\right.
\end{equation}
\end{thm}
%%%%%%%%%%%%%%%%%%%%%%%%%%%5

\section{The Borel transform of WKB solutions}
Let $\psi_{\pm, B}$ denote the Borel transform of $\psi_{\pm}$. Explicitly, $\psi_{\pm}$ have the forms
\begin{align*}
\psi_{\pm}&=\eta^{-\frac{1}{2}}x^{-\frac{1}{4}}\left(1-\frac{5}{32}\eta^{-2}x^{-3}
-\frac{1105}{2048}\eta^{-4}x^{-6}-\cdots\right)^{-\frac{1}{2}}\\ %\label{2.2.1:psipm1}
&\hspace{10mm}\times
\exp\left(\pm\left(\eta\frac{2}{3} x^{\frac{3}{2}}+\eta^{-1}\frac{5}{48}x^{-\frac{3}{2}}+
\eta^{-3}\frac{1105}{9216}x^{-\frac{9}{2}}+\cdots\right)\right).
\end{align*}
Setting 
\[
A=-\frac{5}{32}x^{-3}
-\frac{1105}{2048}\eta^{-2}x^{-6}-\cdots
\]
and 
\[
B=\frac{5}{48}x^{-\frac{3}{2}}+
\eta^{-2}\frac{1105}{9216}x^{-\frac{9}{2}}+\cdots,
\]
we may rewrite them as
\begin{align*}
&\eta^{-\frac{1}{2}}x^{-\frac{1}{4}}(1+\eta^{-2}A(x,\eta))^{-\frac{1}{2}}\exp\left(\pm \eta\frac{2}{3}x^{\frac{3}{2}}\right)\exp(\pm\eta^{-1}B(x,\eta))\\
&=\eta^{-\frac{1}{2}}x^{-\frac{1}{4}}\exp\left(\pm \eta\frac{2}{3}x^{\frac{3}{2}}\right)\left(1-\frac{1}{2}\eta^{-2}A-\frac{3}{8}\eta^{-4}A^{2}+\cdots\right)\\
&\hspace{50mm}\times\left(1\pm \eta^{-1}B+\eta^{-2}\frac{1}{2}B^{2}
\pm\cdots\right)\\
&=\eta^{-\frac{1}{2}}x^{-\frac{1}{4}}\exp\left(\pm \eta\frac{2}{3}x^{\frac{3}{2}}\right)
(1+b_{1}^{\pm}\eta^{-1}x^{-\frac{3}{2}}+b_{2}^{\pm}(\eta^{-1}x^{-\frac{3}{2}})^{2}+\cdots)
\end{align*}
with some constants $b_{j}^{\pm}$. Then $\psi_{\pm, B}$ can be written in the form
\begin{align}\label{4:psipmB-form}
\psi_{\pm,B}(x,y)&=\frac{1}{x\sqrt{\pi}}\left(\frac{y}{x^{\frac{3}{2}}}\pm\frac{2}{3}\right)^{-\frac{1}{2}} 
\left\{1+\frac{\sqrt{\pi}{b}_{1}^{\pm}}{\Gamma(\frac{3}{2})}\left(\frac{y}{x^{\frac{3}{2}}}\pm\frac{2}{3}\right)
+\cdots\right.\\ \nonumber
&\hspace{35mm}\cdots+\left.\frac{\sqrt{\pi}{b}_{j}^{\pm}}{\Gamma(j+\frac{1}{2})}\left(\frac{y}{x^{\frac{3}{2}}}\pm\frac{2}{3}\right)^{j}+\cdots\right\}.
\end{align}
We introduce a new variable $\displaystyle s=\frac{3y}{4x^{\frac{3}{2}}}+\frac{1}{2}$. Then we may rewrite $\psi_{\pm,B}$
as follows:
\begin{align}\label{4:psi+B}
\psi_{+,B}&=\frac{\sqrt{3}}{2\sqrt{\pi}x}s^{-\frac{1}{2}}
\left(1+\tilde{b}_{1}^{+}\frac{4}{3}s+\cdots\right),\\ \label{4:psi-B}
\psi_{-,B}&=\frac{\sqrt{3}}{2\sqrt{\pi}x}(s-1)^{-\frac{1}{2}}\left(1+\tilde{b}_{1}^{-}\frac{4}{3}(s-1)+\cdots\right).
\end{align}
Here we take the branches as $s^{1/2}>0$ if $s>0$ and $(s-1)^{1/2}>0$ if $s>1$. %In addition, $\arg (s-1)^{1/2}=-\pi/2$ for $s<1$.

On the other hand, it follows from the definition of the Borel transform 
that $\psi_{\pm,B}$ satisfies the formal Borel transform of \eqref{2:Airy-system}:
\begin{equation}\label{4:psipmB}
\left\{
\begin{split}
\ & \left(-\frac{\partial^{2}}{\partial x^{2}}+\frac{\partial^{2}}{\partial y^{2}}x\right)\psi_{\pm, B}=0\\
 & \left(2x\frac{\partial}{\partial x}-3\frac{\partial}{\partial y}(-y)-1\right)\psi_{\pm,B}=0,
\end{split}\right.
\end{equation}
or, equivalently,
\begin{equation}\label{4:psipmB-bis}
\left\{\ 
\begin{split}
\ \ &\left(-\frac{\partial^{2}}{\partial x^{2}}+x\frac{\partial^{2}}{\partial y^{2}}\right)\psi_{\pm,B}=0,\\
\ \ &\left(2x\frac{\partial}{\partial x}+3y\frac{\partial}{\partial y}+2\right)\psi_{\pm,B}=0.
\end{split}\right.
\end{equation}
This system coincides with \eqref{3:sys-for-g}. Hence Theorem \ref{3:thm-for-g} yields 
\begin{prp}\label{4:wkb-g}
The Borel transforms $\psi_{\pm,B}$ of the WKB solutions $\psi_{\pm}$ can be
written as linear combinations of any two of $g_{j}$'s. 
Here $g_{j}$ $(j=1,2,3)$ are defined in Theorem \ref{3:thm-for-g}. 
Especially, $\psi_{\pm,B}$ are algebraic functions.
\end{prp}
\begin{rem}
The explicit forms of $\psi_{\pm,B}$ are given by using hypergeometric
functions in \cite{akt0, KT}:
\begin{equation}\label{4:psiBHG}
\left\{
\begin{split} \quad\psi_{+,B}&=\frac{\sqrt{3}}{2\sqrt{\pi}\,x}s^{-\frac{1}{2}}
{}_{2}F_{1}\left(\frac{1}{6},\frac{5}{6},\frac{1}{2};s\right), \\
\quad\psi_{-,B}&=\frac{\sqrt{3}\,i}{2\sqrt{\pi}\,x}(1-s)^{-\frac{1}{2}}
{}_{2}F_{1}\left(\frac{5}{6},\frac{1}{6},
\frac{1}{2};1-s\right).
\end{split}\right.
\end{equation}
%The parameters used in these expressions appear in
%the Schwarz list. This fact also proves that $\psi_{\pm,B}$ being algebraic.
The classical connection formulas for the hypergeometric functions were used in \cite{akt0, KT}  
for the derivation of the Voros connection formula for the Airy equation. 
In the present article, we do not utilize \eqref{4:psiBHG}.
\end{rem}
We set 
\[
g=\frac{X}{x\,s^{1/2}(1-s)^{1/2}}
\]
in \eqref{3:cubic1}. Then we have the following cubic equation for $X$:
\begin{equation}\label{cubicforX}
16 X^{3}-3 X-s^{1/2}(1-s)^{1/2}=0.
\end{equation}
If $s=0$, we have $X=0, \,\pm\!\sqrt{3}/4$ and we find three roots $X_{1}, X_{2}, X_{3}$
of \eqref{cubicforX}
near $s=0$ with expansions
\begin{equation}\label{4:Xjat0}
\left\{
\begin{split}
\ X_{1}&=\frac{\sqrt{3}}{4}+\frac{1}{6}s^{1/2}-\frac{1}{6\sqrt{3}}s+\cdots,\\
X_{2}&=-\frac{\sqrt{3}}{4}+\frac{1}{6}s^{1/2}+\frac{1}{6\sqrt{3}}s+\cdots,\\
X_{3}&=-\frac{1}{3}s^{1/2}-\frac{5}{162}s^{3/2}-\cdots.
\end{split}\right.
\end{equation}
We take the analytic continuation of $X_{j}$ to $s=1$. 
If $s=1$, we also have $X=0, \,\pm\!\sqrt{3}/4$. For $s=1/2$, we have a double root $X=-1/4$ 
and a simple root $X=1/2$. The expansion of the branches of $X$ which merge at $s=1/2$ are
\begin{equation}\label{4:Xat1/2}
X=-\frac{1}{4}\pm\frac{1}{2\sqrt{3}}\left(s-\frac{1}{2}\right)+
\frac{1}{18}\left(s-\frac{1}{2}\right)^{2}\pm\cdots.
\end{equation}
The graphs of $Y=16 X^{3}-3 X-s^{1/2}(1-s)^{1/2}$ for $s=0,\, 1/2,\, 1$ are shown in
Figures 4.1--4.3. Since the coefficients of $(s-1/2)$ of \eqref{4:Xat1/2} do not vanish, two roots 
$X_{2}$ and $X_{3}$ pass each other at $s=1/2$ and
the coefficients of 
the leading terms of the expansions of $(X_{1},X_{2},X_{3})$ at $s=1$ are
$(\sqrt{3}/4,0,-\sqrt{3}/4)$. 

\

\includegraphics[width=40mm]{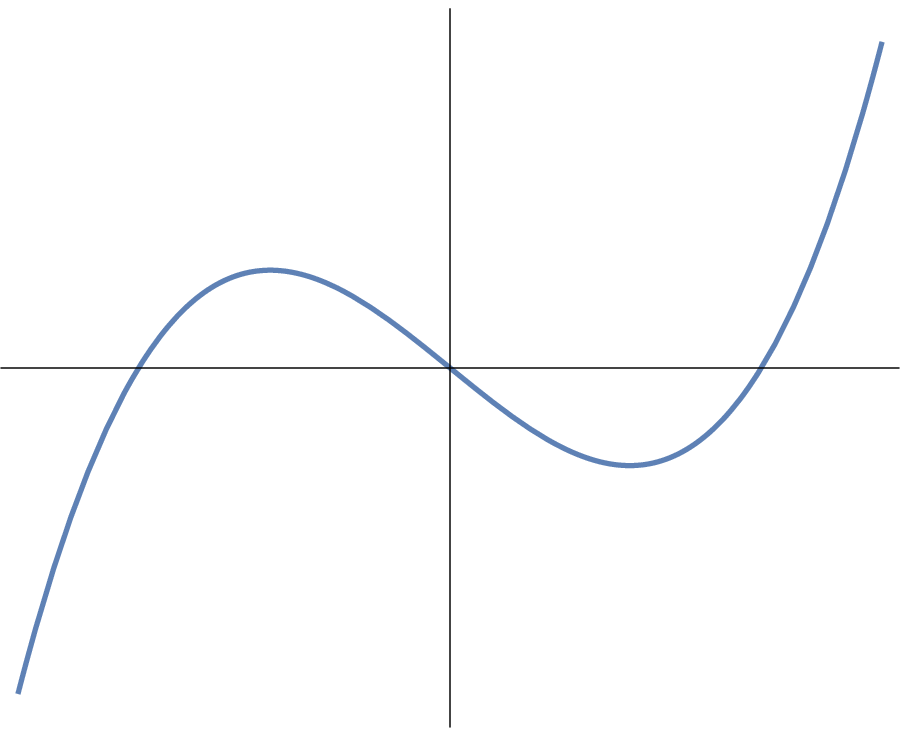}\quad
\includegraphics[width=40mm]{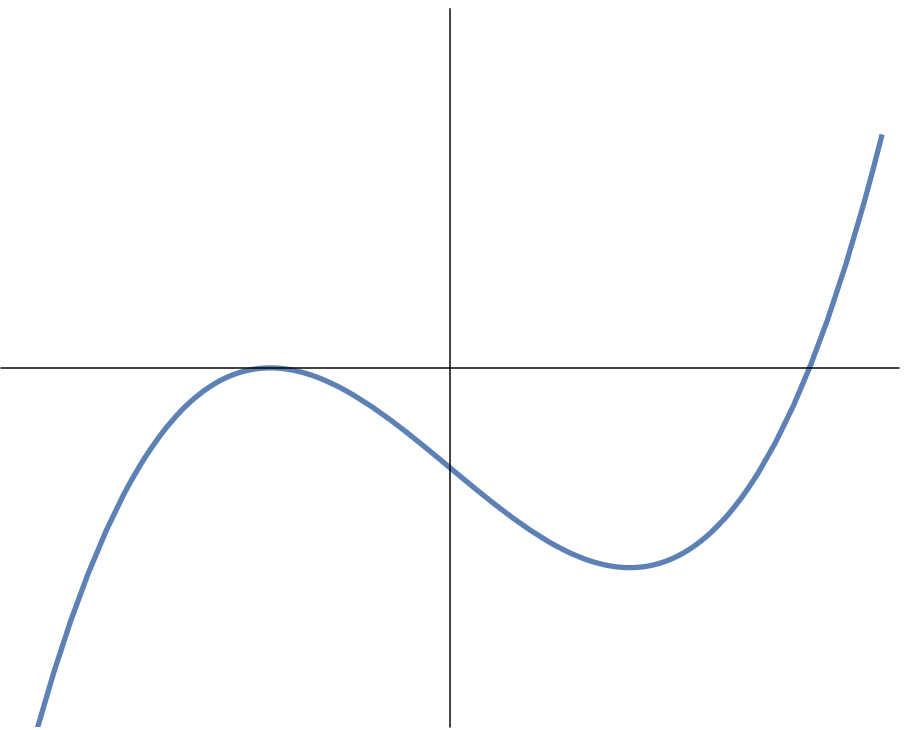}\quad
\includegraphics[width=40mm]{s01t.eps}

\vspace{-23mm}
{\scriptsize $\displaystyle -\frac{\sqrt{3}}{4}$}\hspace{24mm}{\scriptsize $\displaystyle \frac{\sqrt{3}}{4}$}
\hspace{13mm}{\small $X_{2}=X_{3}$ \hspace{14mm}{\scriptsize $\displaystyle \frac{1}{2}$}
\hspace{6mm}
{\scriptsize $\displaystyle -\frac{\sqrt{3}}{4}$}\hspace{24mm}{\scriptsize $\displaystyle \frac{\sqrt{3}}{4}$}

\vspace{1mm}
\hspace*{5mm}{\small$X_{2}$}\hspace{6.5mm}{\small$X_{3}$}\hspace{14mm}{\small$X_{1}$
\hspace{13mm}{\scriptsize$\displaystyle-\frac{1}{4}$}\hspace{22mm}{\small$X_{1}$}
\hspace{8.5mm}{\small$X_{3}$\hspace{6.5mm}{\small$X_{2}$\hspace{13mm}{\small$X_{1}$}

\vspace{-9mm}
\hspace*{20.5mm}$O$\hspace{41mm}$O$\hspace{41mm}$O$

\vspace{18mm}
\hspace*{9mm} Fig. 4.1. $s=0$\hspace{19mm} Fig. 4.2. $s=1/2$
\hspace{19mm} Fig. 4.3. $s=1$

\

\noindent Replacing $s$ by $1-s$ in the expansions of $X_{j}$'s at $s=0$ and comparing 
the leading coefficients, we have
\begin{equation}\label{4:Xjat1}
\left\{
\begin{split}
\ X_{1}&=\frac{\sqrt{3}}{4}+\frac{1}{6}(1-s)^{1/2}-\frac{1}{6\sqrt{3}}(1-s)+\cdots,\\
X_{2}&=-\frac{1}{3}(1-s)^{1/2}-\frac{5}{162}(1-s)^{3/2}-\cdots,\\
X_{3}&=-\frac{\sqrt{3}}{4}+\frac{1}{6}(1-s)^{1/2}+\frac{1}{6\sqrt{3}}(1-s)+\cdots.
\end{split}\right.
\end{equation}
We set
\[
g_{j}=\frac{X_{j}}{xs^{1/2}(1-s)^{1/2}} \quad(j=1,2,3).
\]
Then \eqref{4:Xjat0} yields the expansion of $g_{j}$ at $s=0$:
\begin{align*}
g_{1}&=\frac{1}{xs^{1/2}}\left(\frac{\sqrt{3}}{4}+\frac{1}{6}s^{1/2}+\frac{5}{24\sqrt{3}}s+\cdots\right),\\
g_{2}&=\frac{1}{xs^{1/2}}\left(-\frac{\sqrt{3}}{4}+\frac{1}{6}s^{1/2}-\frac{5}{24\sqrt{3}}s+\cdots\right), \quad g_{3}=-g_{1}-g_{2}.
\end{align*}
Using Proposition \ref{4:wkb-g} and comparing the leading terms of these expansions with  \eqref{4:psi+B}, we have
\begin{equation}\label{4:psi+B-g}
\psi_{+,B}=\frac{1}{\sqrt{\pi}}(g_{1}-g_{2}).
\end{equation}
Similarly, \eqref{4:Xjat1} yields the expansion of $g_{j}$ at $s=1$:
\begin{align*}
g_{1}&=\frac{1}{x(1-s)^{1/2}}\left(\frac{\sqrt{3}}{4}
+\frac{1}{6}(1-s)^{1/2}+\frac{5\sqrt{3}}{24}
(1-s)+\cdots\right),\\
g_{2}&=\frac{1}{x}\left(-\frac{1}{3}+\frac{16}{81}(1-s)+\cdots\right),\\
g_{3}&=\frac{1}{x(1-s)^{1/2}}\left(-\frac{\sqrt{3}}{4}+\frac{1}{6}(1-s)^{1/2}
-\frac{5\sqrt{3}}{24}(1-s)+\cdots\right).
\end{align*}
We compare the leading terms with  \eqref{4:psi-B}. Then we obtain
$\psi_{-,B}={i}(g_{1}-g_{3})/{\sqrt{\pi}}\, (={i}(2g_{1}+g_{2})/{\sqrt{\pi}})$.
Thus we have
\begin{thm}\label{4:psipmandgj}
The Borel transforms $\psi_{\pm,B}$ of the WKB solutions $\psi_{\pm}$ to the
Airy equation are expressed in terms of three branches $g_{j}$'s $($specified as above$)$
of the algebraic function
$g$ defined by \eqref{3:cubic1} 
as follows{\rm :}
\[
\psi_{+,B}=\frac{1}{\sqrt{\pi}}(g_{1}-g_{2}),\quad
\psi_{-,B}=\frac{i}{\sqrt{\pi}}(g_{1}-g_{3}).
\]
These relations can be also written in the form
\[
\psi_{+,B}=\frac{1}{\sqrt{\pi}}\Delta_{-\frac{2}{3}x^{3/2}}g_{2}(x,y),\quad
\psi_{-,B}=\frac{i}{\sqrt{\pi}}\Delta_{\frac{2}{3}x^{3/2}}g_{3}(x,y).
\]
Here $\Delta_{\alpha}g(x,y)$ designates the discontinuity of $g(x,y)$ at $y=\alpha$.
\end{thm}
Note that if $x\neq 0$, then $\pm 2/3 x^{3/2}$ are singularities of square root-type of $\psi_{+, B}$ and of $\psi_{-, B}$
in $y$-variable.

\section{The Voros connection formula for the Airy equation}
The Stokes curve (cf. \cite{KT}) of \eqref{2:Airy-1} is defined by
\[
{\rm Im}\int_{0}^{x}\sqrt{x}\,dx=0.
\]
It consists of three half-lines
\[
\arg x=0, \pm\frac{2}{3}\pi.
\]
Let ${\mathcal R}_{\rm I}$ and ${\mathcal R}_{\rm II}$ denote the sectors (two of the Stokes regions)
\[
\{\,x\in{\mathbb C}\,|\, -2/3\pi<\arg x<0\,\} \quad \mbox{and}\quad
\{\,x\in{\mathbb C}\,|\, 0<\arg x<2/3\pi\,\},
\]
respectively. Let $\ell_{\pm}(x)$ be half-lines
\[
\left\{\left. \, \mp \frac{2}{3}x^{3/2}+t\,\right|\, t\geq 0\,\right\}
\]
with the positive orientation in the $y$-plane and we set
\begin{equation}\label{5:PsiJ}
\Psi_{\pm}^{J}=\int_{\ell_{\pm}(x)}\psi_{\pm,B}(x,y)\exp(-y\eta)dy \quad\mbox{for}\ \ x\in{\mathcal R}_{J}\quad (J={\rm I, II}).
\end{equation}
Figures 5.1 and 5.2 show ${\mathcal R}_{J}$ and $\ell_{\pm}(x)$ for $x\in{\mathcal R}_{J}$ ($J={\rm I, II}$), respectively.
The wavy lines designate the branch cuts for $\psi_{+,B}$ in $y$-plane. 
Theorem \ref{4:psipmandgj} implies that these functions are well-defined and each $\Psi_{\pm}^{J}$ gives the Borel sum of 
$\psi_{\pm}$ in ${\mathcal R}_{J}$
($J={\rm I, II}$), respectively.

\vspace{5mm}

\includegraphics[width=33mm]{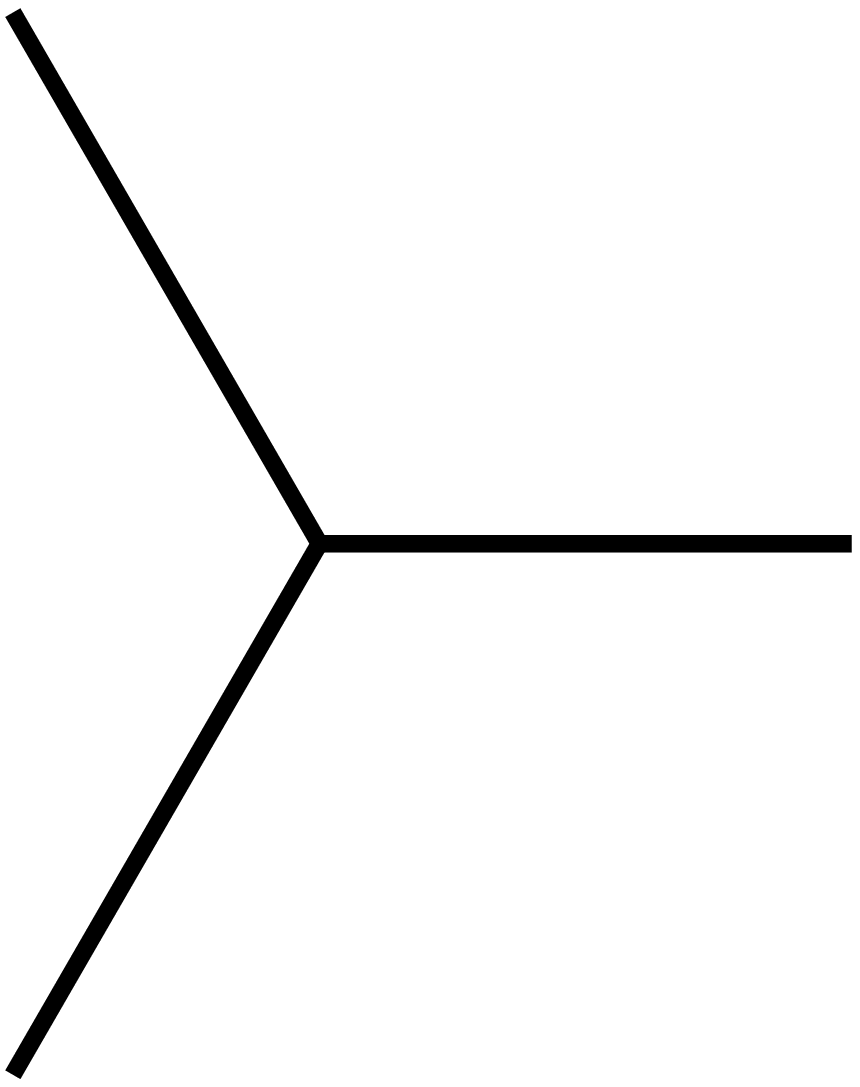}\hspace{8mm}
\includegraphics[width=33mm]{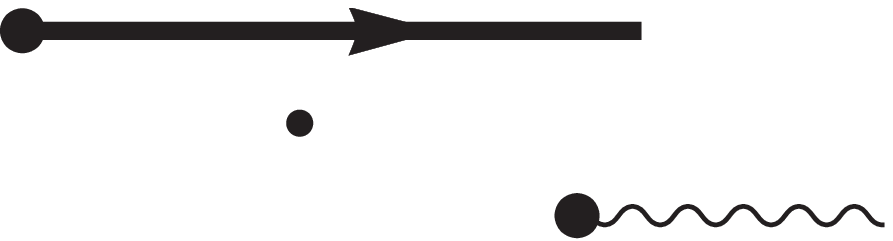}\hspace{8mm}

\vspace{-40mm}
\hspace*{41mm}
\includegraphics[width=33mm]{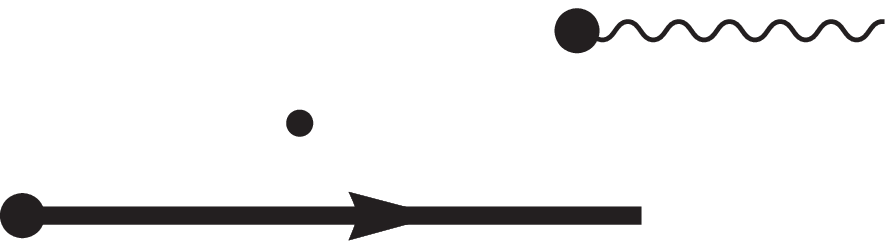}\hspace{8mm}

\vspace{-14mm}\hspace*{60mm}{\small $\frac{2}{3}x^{3/2}$}
\hspace{33mm}{\small $\frac{2}{3}x^{3/2}$}

\vspace{-1mm}\hspace*{117mm}{\small $\ell_{-}(x)$}

\vspace{-1mm}
\hspace{37mm}{\small $-\frac{2}{3}x^{3/2}$}\hspace{32mm}
{\small $-\frac{2}{3}x^{3/2}$}

\hspace*{20mm}{${\mathcal R}_{\rm II}$}\hspace{42mm}{\small $\ell_{+}(x)$}

\vspace{5mm}
\hspace*{43mm} Fig. 5.2.1. $x\in{\mathcal R}_{\rm II}$\hspace{17mm} Fig. 5.2.2. $x\in{\mathcal R}_{\rm II}$

\vspace{-1.5mm}
\hspace*{8mm}$O$

\vspace{8mm}
\hspace*{20mm}{${\mathcal R}_{\rm I}$}\hspace{43mm}{\small $\ell_{+}(x)$}

\vspace{-7mm}\hspace*{37mm}{\small $-\frac{2}{3}x^{3/2}$}\hspace{32mm}
{\small $-\frac{2}{3}x^{3/2}$}

\vspace{9mm}
\hspace*{60mm}{\small $\frac{2}{3}x^{3/2}$}
\hspace{33mm}{\small $\frac{2}{3}x^{3/2}$}

\vspace{-7mm}\hspace*{117mm}{\small $\ell_{-}(x)$}

\vspace{5mm}
 Fig. 5.1. Stokes regions \hspace{9mm} Fig. 5.2.3. $x\in{\mathcal R}_{\rm I}$
\hspace{17mm} Fig. 5.2.4. $x\in{\mathcal R}_{\rm I}$

\vspace{-49mm}\hspace*{65mm}
\includegraphics[width=50mm]{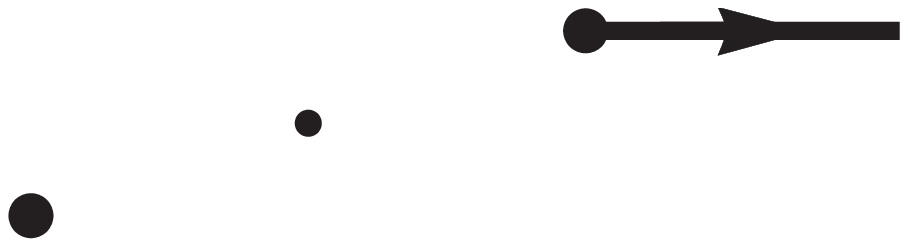}

\vspace{-19mm}
\hspace*{65mm}
\includegraphics[width=50mm]{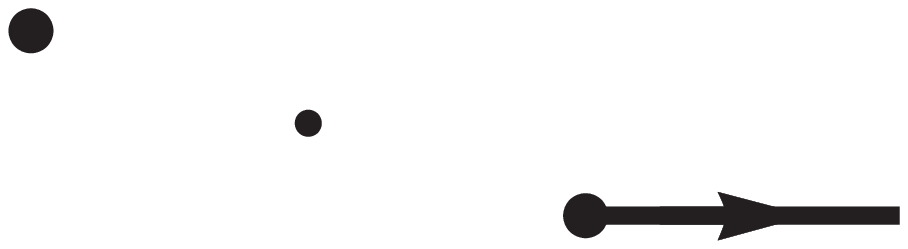}

\vspace{-27mm}
\noindent We take the analytic continuation of $\Psi_{\pm}^{\rm I}$ to ${\mathcal R}_{\rm II}$.
If a point $x\in{\mathcal R}_{\rm I}$ with ${\rm Re}\, x>0$ moves up to ${\mathcal R}_{\rm II}$
across the positive real axis, 
the singular point $y=2/3 x^{3/2}$ of $\psi_{+,B}$ crosses $\ell_{+}(x)$ (see Figures 5.2.3 and 5.2.1).
Hence we have to take a path $\tilde{\ell}_{+}(x)$ of integration shown in Figure 5.3 instead of $\ell_{+}(x)$ for
$x\in{\mathcal R}_{\rm II}$ for the analytic continuation.

\vspace{2mm}
\hspace*{8mm}
\includegraphics[width=45mm]{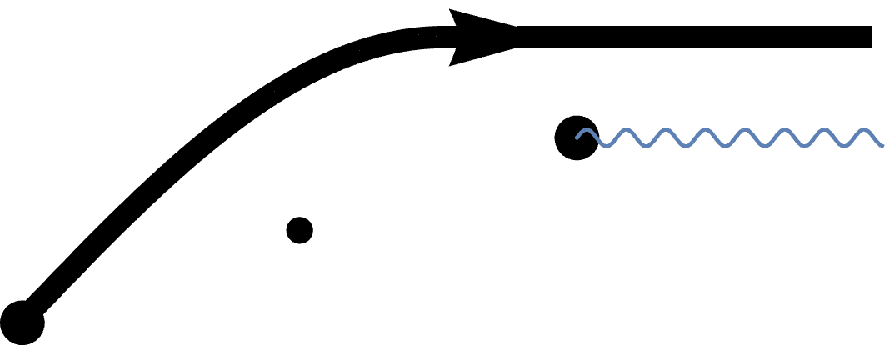}

\vspace{-19mm}
\hspace*{65mm}
\includegraphics[width=45mm]{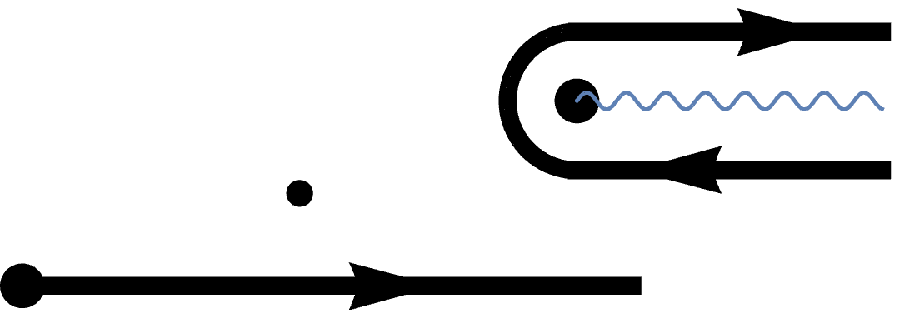}

\vspace{-15mm}
\hspace*{8mm} {\small $\tilde{\ell}_{+}(x)$}\hspace{12mm}{\small$\frac{2}{3}x^{3/2}$}
\hspace{50mm}{$\Gamma$}

\vspace{7mm}
\!\!\!{\small $-\frac{2}{3}x^{3/2}$}\hspace{91mm}{\small $\ell_{+}(x)$}

\vspace{2mm}
\hspace{26mm} Fig. 5.3.\hspace{43mm} Fig. 5.4.

\

\noindent Thus the analytic continuation of $\Psi_{+}^{\rm I}$ has an expression
\[
\Psi_{+}^{\rm I}=\int_{\tilde{\ell}_{+}(x)}\psi_{+,B}(x,y)\exp(-y\eta)dy
\]
for $x\in{\mathcal R}_{\rm II}$.
If we take a path $\Gamma$ surrounding the branch cut (or $\ell_{-}(x)$) clockwise (as shown in Figure 5.4), we have
\[
\tilde{\ell}_{+}(x)\sim\ell_{+}(x)+\Gamma
\]
as the path of integration for $x\in{\mathcal R}_{\rm II}$. Hence we have
\begin{equation}\label{5:psiconti}
\Psi_{+}^{\rm I}=\int_{\ell_{+}(x)}\psi_{+,B}(x,y)\exp(-y\eta)dy
+\int_{\Gamma}\psi_{+,B}(x,y)\exp(-y\eta)dy.
\end{equation}
The first term equals $\Psi_{+}^{\rm II}$. It follows from Theorem \ref{4:psipmandgj} that the second term can be written 
in terms of $g_{j}$'s:
\[
\int_{\Gamma}
\frac{1}{\sqrt{\pi}}(g_{1}-g_{2})\exp(-y\eta)
dy.
\]
Since $g_{2}$ is holomorphic on a neighborhood of $\ell_{-}(x)$ in the $y$-variable, it does not contribute to the value of the integration.
Therefore we have
\begin{equation}\label{5:psiconti-2}
\Psi_{+}^{\rm I}=\Psi_{+}^{\rm II}+\frac{1}{\sqrt{\pi}}\int_{\Gamma}g_{1}\exp(-y\eta)dy.
\end{equation}
The second term of the right-hand side equals
\[
-\frac{1}{\sqrt{\pi}}\int_{\ell_{-}(x)}(\Delta_{\frac{2}{3}x^{3/2}}g_{3})\exp(-y \eta)dy.
\]
Hence Theorem \ref{4:psipmandgj} yields
\[
\Psi_{+}^{\rm I}=\Psi_{+}^{\rm II}-\frac{1}{i}\int_{\ell_{-}(x)}\psi_{-,B}(x,y)\exp(-y\eta)dy,
\]
namely,
\[
\Psi_{+}^{\rm I}=\Psi_{+}^{\rm II}+i\Psi_{-}^{\rm II}.
\]
On the other hand, the singularity $y=-2/3 x^{3/2}$ does not meet $\ell_{-}(x)$ when $x$ moves from ${\mathcal R}_{\rm I}$
to ${\mathcal R}_{\rm II}$ across the positive real axis. This implies
 $\Psi_{-}^{\rm I}=\Psi_{-}^{\rm II}$.
 Hence we have obtained the Voros connection formula for the WKB solutions to the Airy equation
 without using any knowledge concerning the hypergoemetric functions:
 \begin{equation*}
\left\{
\begin{split}
\ \Psi_{+}^{\rm I}&=\Psi_{+}^{\rm II}+i\,\Psi_{-}^{\rm II},\\
\Psi_{-}^{\rm I}&=\Psi_{-}^{\rm II}.
\end{split}\right.
\end{equation*}

%%%%%%%%%%%%%%%%%%%%%%%%%%%%%%%%%%%%
%%%%%%%%%%%%%%%%%%%%%%%%%%%%%%%%%%%%
%%%%%%%%%%%%%%%%%%%%%%%%%%%%%%%%%%%

\section{Relation between WKB solutions and Airy functions}\label{6:section}
%%%%%%%%%%
In this section, we relate the Airy functions and the WKB solutions. 
We employ the same notation as in the previous sections. 
%Discussions given there yields 
\begin{thm}\label{6:AiBi-Psipm}
The Airy functions ${\rm Ai}$ and ${\rm Bi}$ are expressed in terms of the Borel sums $\Psi_{\pm}^{\rm I}$ of WKB solutions 
$\psi_{\pm}$ to the Airy equation in ${\mathcal R}_{\rm I}$ as follows{\rm :}
\begin{equation}\label{6:AiBi-I}
\left\{\ 
\begin{split}
{\rm Ai}(\eta^{\frac{2}{3}}x)&=\eta^{\frac{1}{3}}\frac{1}{2\sqrt{\pi}}\Psi_{-}^{\rm I},\\
{\rm Bi}(\eta^{\frac{2}{3}}x)&=\eta^{\frac{1}{3}}\frac{1}{\sqrt{\pi}}\Psi_{+}^{\rm I}-
\eta^{\frac{1}{3}}\frac{i}{2\sqrt{\pi}}\Psi_{-}^{\rm I}.
\end{split}\right.\end{equation}
\end{thm}
\begin{proof}
We choose the path $\gamma_{j}$ of integration \eqref{3:varphij} more carefully. If $x\in {\mathcal R}_{\rm I}$, we take
$\gamma_{1}$ (resp. $\gamma_{2}$) as the steepest descent path of the phase function of the integral
\eqref{3:varphij} passing through the saddle point $\sqrt{x}$ (resp. $-\sqrt{x}$). 
Here we take the branch as $|\arg\sqrt{x\,}\,|<\pi/3$ for $x\in {\mathcal R}_{\rm I}\cup{\mathcal R}_{\rm II}$.
%%%%%%%
%%%%%%%
As a set, $\gamma_{1}$ (resp. $\gamma_{2}$) is included in
\[
\left\{\,t\in{\mathbb C}\,\left|\, {\rm Im}\left(\frac{t^{3}}{3}+x t+\frac{2}{3}x^{3/2}\right)=0, 
\, {\rm Re}\left(\frac{t^{3}}{3}-x t+\frac{2}{3}x^{3/2}\right)\le 0\,\right.\right\}
\]
\[
\left( {\rm resp.} \left\{\,t\in{\mathbb C}\,\left|\, {\rm Im}\left(\frac{t^{3}}{3}-x t-\frac{2}{3}x^{3/2}\right)=0, 
\, {\rm Re}\left(\frac{t^{3}}{3}-x t-\frac{2}{3}x^{3/2}\right)\le 0\,\right.\right\}\right).
\]
We go back to \eqref{3:int-1} and specify the branch of $g$ more precisely.
Let us divide $\gamma_{1}$ into two parts by cutting it at the saddle point $\sqrt{x}$. We denote by
$\gamma_{1}^{-}$ (resp. $\gamma_{1}^{+}$) the lower (resp. upper) part.
Let $c_{1}^{\pm}$ be the image of $\gamma_{1}^{\pm}$ by the mapping $t\,\mapsto\,y$.
Then $c_{1}^{-}$ (resp. $c_{1}^{+}$) coincides with $-\ell_{-}(x)$ (resp. $\ell_{-}(x)$).
If $y\in c_{1}^{-}$ (resp. $y\in c_{1}^{+}$) is close to $2/3x^{3/2}$, the branch of the root $t$ 
of \eqref{3:eqfort} should be taken as
\[
t=x^{1/2}\left(1+\frac{2\sqrt{3}}{3}(1-s)^{1/2}+\cdots\right)\quad \left(\mbox{resp.} \ \ 
t=x^{1/2}\left(1-\frac{2\sqrt{3}}{3}(1-s)^{1/2}+\cdots\right)\right).
\]
Here we choose the branch as $(1-s)^{1/2}=(s-1)^{1/2}e^{-\pi i/2}$ for $s>1$. 
Hence the branch of $\displaystyle g=\frac{1}{x-t^{2}}$ 
should be taken as $g_{1}$ on $c_{1}^{-}$ (resp. $g_{3}$ on $c_{1}^{+}$). 
%%%%%%%%%%%%%%%%%%%%%%%%%%%%%%%%%%%%%
By the definition of $g_{j}$, \eqref{3:int-1} for $j=1$ should be understood as
\begin{align*}
\varphi_{1}&=\int_{c_{1}^{-}}g_{1}(x,y)\exp(-y\eta)dy+\int_{c_{1}^{+}}g_{3}(x,y)\exp(-y\eta)dy\\
&=\int_{-\ell_{-}(x)}g_{1}(x,y)\exp(-y\eta)dy+\int_{\ell_{+}(x)}g_{3}(x,y)\exp(-y\eta)dy\\
&=-\int_{\ell_{-}(x)}(g_{1}(x,y)-g_{3}(x,y))\exp(-y\eta)dy.
\end{align*}
Using Theorem \ref{4:psipmandgj}, we have
\[
\varphi_{1}=-\frac{\sqrt{\pi}}{i}\int_{\ell_{-}(x)}\psi_{-,B}(x,y)\exp(-y\eta)dy.
\]
Thus the first equation of \eqref{3:AiBi-varphi} shows
\[
{\rm Ai}(\eta^{2/3}x)=-\frac{\eta^{\frac{1}{3}}}{2\pi i}\cdot\frac{\sqrt{\pi}}{i}\int_{\ell_{-}(x)}\psi_{-,B}(x,\eta)\exp(-y\eta)dy
=\eta^{\frac{1}{3}}\frac{1}{2\sqrt{\pi}}\Psi_{-}^{\rm I}.
\]
Next we divide $\gamma_{2}$ into two parts at the saddle point $-\sqrt{x}$. We denote by $\gamma_{2}^{-}$
(resp. $\gamma_{2}^{+}$) the right (resp. left) part. Let $c_{2}^{-}$ (resp. $c_{2}^{+}$) be the image of
$\gamma_{2}^{-}$ (resp. $\gamma_{2}^{+}$) by the mapping $t\,\mapsto\,y$. Then $c_{2}^{-}$ (resp. $c_{2}^{+}$)
coincides with $-\ell_{+}(x)$ (resp. $\ell_{+}(x)$).
%%%%%%%%%%%%%%%%%%%%%%%%%%%%
If $y\in c_{2}^{-}$ (resp. $y\in c_{2}^{+}$) is close to $-2/3 x^{3/2}$, the branch of the root $t$ 
of \eqref{3:eqfort} should be taken as
\[
t=x^{1/2}\left(-1+\frac{2\sqrt{3}}{3}s^{1/2}+\cdots\right)\quad
\left(\mbox{resp.} \ \ t=x^{1/2}\left(-1-\frac{2\sqrt{3}}{3}s^{1/2}+\cdots\right)\right).
\]
Hence the branch of $\displaystyle g=\frac{1}{x-t^{2}}$ 
must be taken as $g_{1}$ on $c_{1}^{-}$ (resp. $g_{2}$ on $c_{1}^{+}$). 
%%%%%%%%%%%%%%%%%%%%%%%%%%%%
Then we have
\begin{align*}
\varphi_{2}&=\int_{c_{2}^{-}}g_{1}(x,y)\exp(-y\eta)dy+\int_{c_{2}^{+}}g_{2}(x,y)\exp(-y\eta)dy\\
&=-\int_{\ell_{+}(x)}(g_{1}(x,y)-g_{2}(x,y))\exp(-y\eta)dy\\
&=-\sqrt{\pi}\int_{\ell_{+}(x)}\psi_{+,B}(x,y)\exp(-y\eta)dy.
\end{align*}
Hence the second equation in \eqref{3:AiBi-varphi} yields
\begin{align*}
{\rm Bi}(\eta^{2/3}x)&=\frac{\eta^{\frac{1}{3}}}{2\pi}(-\varphi_{1}(x,\eta)-2\varphi_{2}(x,\eta))\\
&=\frac{\eta^{\frac{1}{3}}}{2\pi}\frac{\sqrt{\pi}}{i}\int_{\ell_{-}(x)}\psi_{-,B}(x,y)\exp(-y\eta)dy\\
& \hspace{25mm} +\frac{\eta^{\frac{1}{3}}}{\pi}\int_{\ell_{+}(x)}\sqrt{\pi}
\psi_{+,B}(x,y)\exp(-y\eta)dy\\
&=\frac{\eta^{\frac{1}{3}}}{\sqrt{\pi}}\Psi_{+}^{\rm I}-\frac{\eta^{\frac{1}{3}}}{2\sqrt{\pi}}i\Psi_{-}^{\rm I}.
\end{align*}
This completes the proof.
\end{proof}
\begin{rem}
Using the Voros connection formula, we obtain
\begin{equation*}
\left\{\ 
\begin{split}
{\rm Ai}(\eta^{\frac{2}{3}}x)&=\eta^{\frac{1}{3}}\frac{1}{2\sqrt{\pi}}\Psi_{-}^{\rm II},\\
{\rm Bi}(\eta^{\frac{2}{3}}x)&=\eta^{\frac{1}{3}}\frac{1}{\sqrt{\pi}}\Psi_{+}^{\rm II}+
\eta^{\frac{1}{3}}\frac{i}{2\sqrt{\pi}}\Psi_{-}^{\rm II}.
\end{split}\right.
\end{equation*}
Conversely, $\Psi_{\pm}^{\rm J}$ are written in terms of ${\rm Ai}$ and ${\rm Bi}$ as follows:
\begin{equation*}
\left\{\ 
\begin{split}
\Psi_{+}^{\rm I}&=\eta^{-\frac{1}{3}}\sqrt{\pi}\, i\, {\rm Ai}(\eta^{\frac{2}{3}}x)
+\eta^{-\frac{1}{3}}\sqrt{\pi}\,{\rm Bi}(\eta^{\frac{2}{3}}x),\\
\Psi_{-}^{\rm I}&=\eta^{-\frac{1}{3}}2\sqrt{\pi}\,{\rm Ai}(\eta^{\frac{2}{3}}x),
\end{split}\right.
\end{equation*}
\begin{equation*}
\left\{\ 
\begin{split}
\Psi_{+}^{\rm II}&=-\eta^{-\frac{1}{3}}\sqrt{\pi}\, i\, {\rm Ai}(\eta^{\frac{2}{3}}x)
+\eta^{-\frac{1}{3}}\sqrt{\pi}\,{\rm Bi}(\eta^{\frac{2}{3}}x),\\
\Psi_{-}^{\rm II}&=\eta^{-\frac{1}{3}}2\sqrt{\pi}\,{\rm Ai}(\eta^{\frac{2}{3}}x).
\end{split}\right.
\end{equation*}
\end{rem}
Watson's lemma and
the expressions given above reproduce the classical
asymptotic formulas for the Airy functions (cf. \cite{O}):
\begin{equation*}
\left\{\ 
\begin{split}
{\rm Ai}(\eta^{\frac{2}{3}}x)&\sim\eta^{\frac{1}{3}}\frac{1}{2\sqrt{\pi}}\psi_{-},\\
{\rm Bi}(\eta^{\frac{2}{3}}x)&\sim\eta^{\frac{1}{3}}\frac{1}{\sqrt{\pi}}\psi_{+}-
\eta^{\frac{1}{3}}\frac{i}{2\sqrt{\pi}}\psi_{-}
\end{split}\right.\quad (x\in{\mathcal R}_{\rm I},\ \eta\rightarrow\infty),
\end{equation*}
\begin{equation*}
\left\{\ 
\begin{split}
{\rm Ai}(\eta^{\frac{2}{3}}x)&\sim\eta^{\frac{1}{3}}\frac{1}{2\sqrt{\pi}}\psi_{-},\\
{\rm Bi}(\eta^{\frac{2}{3}}x)&\sim\eta^{\frac{1}{3}}\frac{1}{\sqrt{\pi}}\psi_{+}+
\eta^{\frac{1}{3}}\frac{i}{2\sqrt{\pi}}\psi_{-}
\end{split}\right.\quad (x\in{\mathcal R}_{\rm II},\ \eta\rightarrow\infty).
\end{equation*}
Explicit forms of $\psi_{\pm}$ are written as follows:
\[
\psi_\pm=\frac{{e^{\pm\frac{2}{3}x^{3/2}\eta}}}{2\pi}\sum_{n=0}^\infty 
{\eta^{-n-\frac12}}
\left(\pm\frac{3}{4}\right)^n\frac{\Gamma\big(n+\frac{1}{6}\big)\Gamma\big(n+\frac{5}{6}\big)}{n!} x^{-\frac{3}{2}n-\frac{1}{4}}.
\]

\section{Some generalization}
We consider the following integral:
\begin{equation}\label{7:Pearcey-int}
v=\int\exp\left(\eta\left(t^{4}+x_{2} t^{2}+x_{1} t\right)\right)dt.
\end{equation}
Here $x_{1}, x_{2}$ are complex variables and the path of integration is taken suitably.
This is a natural extension of \eqref{3:varphij} to two-variable case. It is called the
Pearcey integral (\cite{O}) with the large parameter $\eta$. Most parts of the discussions developed in \S\S 2--6
can be generalized to the system of partial differential equations that characterizes this integral. 
We review some results concerning this system without proof. Details will be given in \cite{asu}.

We can easily see that $\psi=v$ satisfies 
the system of partial differential equations
\begin{equation}\label{7:Pearcey0}
\left\{\ \ 
\begin{split}
&\left(4\partial_{1}^{3}+2x_{2}\eta^{2}\partial_{1}+x_{1}\eta^{3}\right)\psi=0,\\
&\left(\eta\partial_{2}-\partial_{1}^{2}\right)\psi=0,
\end{split}
\right.
\end{equation}
where we set $\partial_{1}=\partial/\partial x_{1}$,
$\partial_{2}=\partial/\partial x_{2}$. 
WKB solutions to this system is constructed in \cite{a, h} and the connection problem of the solutions was
discussed in \cite{h}. We employ another system of partial differential equations. We set 
\begin{align}\label{7:pj}
{P}_1&=4\partial_1 \partial_2+2 \eta  x_2 \partial_1 +\eta ^2 x_1,\\
{P}_2&=4 \partial_2^2+\eta  x_1 \partial_1 +2\eta  x_2 \partial_2 +\eta,\\
{P}_3&=\eta\partial_{2}- \partial_{1}^{2}.
\end{align}
It is easy to see that $\psi=v$ is a solution to the system (cf. \cite{OK})
\begin{equation}\label{7:Pearcey-1}
\left\{
\begin{split}
P_{1}\psi&=0,\\
P_{2}\psi&=0,\\
P_{3}\psi&=0.
\end{split}
\right.
\end{equation}
In fact, we have
\[
P_{1}v=\eta\int\partial_{t}\exp\left(\eta\left(t^{4}+x_{2} t^{2}+x_{1} t\right)\right)dt=0,
\]
\[
P_{2}v=\eta\int\partial_{t}(\left(t  \exp\left(\eta\left(t^{4}+x_{2} t^{2}+x_{1} t\right)\right)\right)dt=0,
\]
\[
P_{3}v=\eta^{2}\int\left(t^{2}-t^{2}\right)\exp\left(\eta\left(t^{4}+x_{2} t^{2}+x_{1} t\right)\right)dt=0.
\]
Setting
\[
Q_{1}=4\partial_{1}^{3}+2x_{2}\eta^{2}\partial_{1}+x_{1}\eta^{3}\quad\mbox{and}\quad 
Q_{2}=\eta\partial_{2}-\partial_{1}^{2}(=P_{3}),
\]
we can confirm the following relations:
\begin{align}\label{7:p1-q1q2}
P_{1}&=\eta^{-1}(Q_{1}+4\partial_{1}Q_{2}),\\ \label{7:p2-q1q2}
P_{2}&=\eta^{-2}\partial_{1}Q_{1}+(4\eta^{-2}Q_{2}+8\eta^{-2}\partial_{1}^{2}+2x_{2})Q_{2},\\ \label{7:q1-p1p3}
Q_{1}&=\eta P_{1}-4\partial_{1}P_{3}.
\end{align}
Hence if $\eta\neq 0$ is fixed, \eqref{7:Pearcey-1} is equivalent to \eqref{7:Pearcey0}. We also note that
\[
P_{2}=\eta^{-1}\partial_{1}P_{1}+2(2\eta^{-1}\partial_{2}+x_{2})P_{3}
\]
holds. 
Next we consider $\eta$ as an independent complex variable. 
Then the systems \eqref{7:Pearcey0} and \eqref{7:Pearcey-1} are subholonomic. 
To find another independent differential equation for $v$, 
%%%%%%%%%%%%%%%%%%%%%%%%%%%%%%%
we look at the weighted homogeneity of  \eqref{7:Pearcey-int} 
in $(x_{1}, x_{2}, \eta)$. We can see that
\[
v(\lambda^{3}x_{1},\lambda^{2}x_{2},\lambda^{-4}\eta)=\lambda v(x_{1},x_{2},\eta)
\]
holds for $\lambda\neq 0$ and hence $\psi=v$ is a solution to
\begin{equation}\label{7:homog}
(3x_{1}\partial_{1}+2x_{2}\partial_{2}-4\eta\partial_{\eta}-1)\psi=0,
\end{equation}
where $\partial_{\eta}=\partial/\partial\eta$.
We set 
\[
P_{4}=3x_{1}\partial_{1}+2x_{2}\partial_{2}-4\eta\partial_{\eta}-1
\]
and consider the
system of partial differential equations 
\[
P_{j}\psi=0\quad (j=1,2,3,4).
\]
To be more specific, let $\mathcal D$ denote the Weyl algebra
of the variable $(x_{1},x_{2},\eta)$ and $\mathcal I$ the left ideal in $\mathcal D$ 
generated by $P_{j}$  ($j=1,2,3,4$). 
We can prove the following theorem (cf. \cite{Oaku2}):
\begin{thm}\label{7:M}
Let ${\mathcal M}$ denote the left $\mathcal D$-module defined by ${\mathcal I} :$
\[
{\mathcal M}={\mathcal D}/{\mathcal I}.
\]
Then ${\mathcal M}$ is a holonomic system of rank $3$.
\end{thm}
Thus ${\mathcal M}$ characterizes the 3-dimensional linear subspace spanned by \eqref{7:Pearcey-int}
in the space of analytic functions. Note that there are four valleys of the integral \eqref{7:Pearcey-int} and
hence six infinite paths of integration connecting distinct two valleys. Any three of them are independent, 
which give a basis of the solution space.

Next we construct WKB solutions to ${\mathcal M}$. We set $S=\partial_{1}\psi/\psi$ and $T=\partial_{2}\psi/\psi$.
Since we have  \eqref{7:p1-q1q2} and \eqref{7:p2-q1q2}, we can use \eqref{7:Pearcey0} to find $S, T$. 
That is, we see that $S$ and $T$ should satisfy
\[
4S^{3}+2\eta^{2}x_{2}S+\eta^{3}x_{1}+12S\partial_{1}S+4\partial_{1}^{2}S=0
\]
and
\[
\eta T-\partial_{1}S-S^{2}=0.
\]
We seek formal solutions to these equations of the forms
\[
S=\sum_{k=-1}^{\infty}\eta^{-k}S_{k},\quad 
T=\sum_{k=-1}^{\infty}\eta^{-k}T_{k}.
\]
Putting these expressions into the above equations, we see that 
$S_{k}$ and $T_{k}$ are obtained by
the following 
recursion relations and initial conditions:
\begin{align}\label{7:cubic-eq}
4&S_{-1}^{3}+2x_{2}S_{-1}+x_{1}=0,\\
S_{0}&=-\frac{1}{2}\partial_{1}\log(6S_{-1}^{2}+x_{2}),\\
S_{k}&=-\frac{2}{6S_{-1}^{2}+x_{2}}\left(
\sum_{\underset{ -1\le k_{1},\, k_{2},\, k_{3}<k}{k_{1}+k_{2}+k_{3}=k-2}}
S_{k_{1}}S_{k_{2}}S_{k_{3}}\right.\\ \nonumber
&\hspace{40mm}\left.
+\,3\sum_{\underset{ -1\le k_{1},\, k_{2}<k}{k_{1}+\,k_{2}=k-2}}S_{k_{1}}\partial_{1}S_{k_{2}}
+\partial_{1}^{2}S_{k-2}\right)\ \ (k\geq 1),\\
T_{-1}&=S_{-1}^{2}, \\
T_{k}&=\partial_{1}S_{k-1}+\sum_{j=-1}^{k}S_{j}S_{k-j-1}\ \ 
(k\geq 0).
\end{align}
This construction is the same as that given in \cite{a, h} and hence the 1-form of formal series
\[
\omega=Sdx_{1}+Tdx_{2}
\]
is closed. In these references, a formal solution of the form
\[
\psi=\eta^{-1/2}\exp\left(\int_{(a_{1},a_{2})}^{(x_{1},x_{2})}\omega\right)
\]
is called a WKB solution to \eqref{7:Pearcey0}. Here $(a_{1},a_{2})$ is a suitably fixed point.

Now we consider the WKB solutions to ${\mathcal M}$ of the form
\[
\psi=\eta^{-1/2}\exp\left(\int\omega\right).
\]
In addition to \eqref{7:Pearcey-1} (or \eqref{7:Pearcey0}), 
$\psi$ should satisfy \eqref{7:homog} and hence the choice of the primitive of $\displaystyle \omega
=\sum_{k=-1}^{\infty}\eta^{-k}\omega_{k}=\sum_{k=-1}^{\infty}\eta^{-k}(S_{k}dx_{1}+T_{k}dx_{2})$ is
constrained by this equation (up to genuine additive constants), 
namely, 
\begin{equation}\label{7:primitives}
\left\{
\begin{split}
\int \omega_{0}&=-\frac{1}{2}\log(6S_{-1}^{2}+x_{2}),\\
\int\omega_{k}&=-\frac{1}{4k}(3x_{1}S_{k}+2x_{2}T_{k})\quad(k\neq 0).
\end{split}\right.
\end{equation}
This choice is consistent with the construction of $S_{k}$ and $T_{k}$. 
In fact, we can confirm the first equation of \eqref{7:primitives} by direct computation and the second 
by using the homogeneity of $S_{k}$ and $T_{k}$.
From now on, we take special WKB solutions of the form
\begin{equation}\label{7:WKBsol}
\psi=\frac{1}{\left(\eta(6S_{-1}^{2}+x_{2})\right)^{1/2}}\exp\left(\eta\int\omega_{-1}+\sum_{k=1}^{\infty}\eta^{-k}\int\omega_{k}\right)
\end{equation}
with the primitives given by \eqref{7:primitives}. Let $S_{-1}^{(j)}$ ($j=1,2,3$) denote the three roots of \eqref{7:cubic-eq} and set 
$T_{-1}^{(j)}=(S_{-1}^{(j)})^{2}$. According to this choice, we have three WKB solutions $\psi_{j}$ ($j=1,2,3$) of the form
\eqref{7:WKBsol}.

Let $\psi_{j, B}$ be the Borel transform of $\psi_{j}$  ($j=1,2,3$) and $P_{k,B}$ the formal Borel transform of $P_{k}$
($k=1,2,3,4$) .
The explicit forms of $P_{k,B}$'s are given as follows:
\begin{align*}
P_{1,B}&=4 \partial_{1}\partial_{2} + 2 x_2\partial_{y}\partial_{1}
 + x_1\partial_{y}^2,\\
P_{2,B}&=4 \partial_{2}^2 +  x_1\partial_{y}\partial_{1} + 2 x_2\partial_{y} \partial_{2} +\partial_{y},\\
P_{3,B}&=\partial_{y}\partial_{2}-\partial_{1}^{2},\\
P_{{4,B}}&=3x_{1}\partial_{1}+2x_{2}\partial_{2}-4\partial_{y}(-y)-1\\
(&=3x_{1}\partial_{1}+2x_{2}\partial_{2}+4y\partial_{y}+3).
\end{align*}
Then $P_{k,B}\psi_{j,B}=0$ holds for $j=1,2,3; k=1,2,3,4$. We denote by ${\mathcal I}_{B}$ the left ideal
in ${\mathcal D}_{B}$ generated by $P_{k, B}$ ($k=1,2,3,4$). Here ${\mathcal D}_{B}$ denotes the Weyl
algebra of the variable $(x_{1},x_{2},y)$. Then the following theorem can be proved in a similar manner to
Theorem \ref{7:M}
\begin{thm}\label{7:MB}
Let ${\mathcal M}_{B}$ denote the left ${\mathcal D}_{B}$-module defined by ${\mathcal I}_{B} :$
\[
{\mathcal M}_{B}={\mathcal D}_{B}/{\mathcal I}_{B}.
\]
Then ${\mathcal M}_{B}$ is a holonomic system of rank $3$.
\end{thm}
Thus ${\mathcal M}_{B}$ characterizes the subspace of analytic functions spanned by $\psi_{j,B}$ ($j=1,2,3$).

\

We go back to \eqref{7:Pearcey-int}. We set $t^{4}+x_{2} t^{2}+x_{1} t=-y$ and rewrite \eqref{7:Pearcey-int} as
\[
v=\int \exp(-\eta y)g(x_{1},x_{2},y)dy.
\]
Here $g$ is defined by
\[
g(x_{1},x_{2},y)=\left.\frac{1}{4 t^{3}+2x_{2}t+x_{1}}\right|_{t=t(x_{1},x_{2})}
\]
and the path of integration is suitably modified. The following lemma can be proved in a similar way to
the proof of Lemma \ref{3:Lemma1}.

\begin{lmm}
The function $g$ defined as above satisfies the quadratic equation
\[
(4 x_1^2 x_2(36  y-x_2^2)+16 y(x_2^2-4 y)^2-27 x_1^4)\, g^4+2(-8x_2 y+2 x_2^3+9 x_1^2)\,g^2 -8x_1g+1=0
\]
and it is a solution to the holonomic system ${\mathcal M}_{B}$.
\end{lmm}
We note that the singular locus of ${\mathcal M}_{B}$ coincides with the zero-point set of the leading coefficient
of the above quadratic equation. 
For general $(x_{1},x_{2},y)$, there are four roots $g_{k}$ $(k=1,2,3,4)$ of
the quadratic equation, which satisfy $g_{1}+g_{2}+g_{3}+g_{4}=0$. Looking at the singularity of $g_{k}$,
we find that any three of $g_{k}$'s are linearly independent. Thus we have the following theorem.
\begin{thm}
The Borel transform $\psi_{j,B}$ of the WKB solution $\psi_{j}$ $(j=1,2,3)$ can be
written as a linear combination of any three of $g_{k}$'s. In particular, $\psi_{j,B}$'s are
algebraic and hence they are resurgent.
\end{thm}
We can write down $\psi_{j, B}$ in terms of $g_{k}$'s. Explicit forms will be given in \cite{asu}.

%%%%%%%%%%%%%%%%%%%%%%%%%%%%%%%

%%%%%%%%%%%%%%%%%%%%%%%%%%%%%%

%%%%%%%%%%%%%%%%%%%%%%%%%%%%%%
%%%%%%%%%%%%%%%%%%%%%%%%%%%%%%
\end{document}